\documentclass{amsart}
\usepackage{a4wide, graphicx, amsmath, amsfonts, amsthm, amssymb, latexsym}   

\usepackage[colorlinks]{hyperref}
\hypersetup{linkcolor=blue, urlcolor=blue, citecolor=red}

\let\oldmarginpar\marginpar
\renewcommand\marginpar[1]{\-\oldmarginpar[\raggedleft\footnotesize #1]%
{\raggedright\footnotesize #1}}

\newtheorem{thm}{Theorem}[section]
\newtheorem{lem}[thm]{Lemma}

\newtheorem{cor}[thm]{Corollary}

\newcommand{\proofref}[1]{\noindent {\emph{Proof of Theorem} \ref{#1}.\ }}

\newcommand{\set}[2]{\{#1:#2\}}
\newcommand{\genset}[1]{\langle#1\rangle}

\newcommand{\nat}{\mathbb{N}}

\newcommand{\trans}{\Omega^{\Omega}}
\newcommand{\sym}{\operatorname{Sym}(\Omega)}

\newcommand{\inv}{\operatorname{I}_{\Omega}}

\begin{document}

\title{Sierpi\'nski rank of the symmetric inverse semigroup}

\author{J.T. Hyde and Y. P\'eresse}
\begin{abstract}
We show that every countable set of partial bijections from an infinite set to itself can be obtained as a composition of just two such partial bijections. This strengthens a result by Higgins, Howie, Mitchell and Ru\v{s}kuc stating that every such countable set of partial bijections may be obtained as the composition of two partial bijections and their inverses.
\end{abstract}

\maketitle


\section{Introduction} 
Cayley's Theorem states that every group is isomorphic to a subgroup of the symmetric group $\sym$ of all permutations of some set $\Omega$. In fact, any group $G$ embeds in $\operatorname{Sym}(G)$. 

In this sense, the semigroup-theoretic analogue of $\sym$ is $\trans$, the semigroup of all functions from $\Omega$ to $\Omega$. Every semigroup $S$ is isomorphic to a subsemigroup of $\trans$ for some set $\Omega$ with $|\Omega| \leq |S|+1$.

For inverse semigroups, the corresponding object is  the inverse semigroup $\inv$ of all partial bijections on $\Omega$, i.e. bijections with range and domain a subset of $\Omega$. Every inverse semigroup $S$ embeds into $\inv$ for some set $\Omega$ with $|\Omega|\leq |S|+1$.

The following theorem is a classical result by Sierpi\'nski. 

\begin{thm}[\cite{sierpinski}, Th\'eor\`eme I] \label{transtheorem} Let $\Omega$ be an infinite set. Then every countable subset of $\trans$ is contained in a $2$-generated subsemigroup of $\trans$.
\end{thm}
Because of the property of $\trans$ mentioned above, Theorem \ref{transtheorem} immediately implies that every countable semigroup embeds in a $2$-generated semigroup. 
In light of Theorem \ref{transtheorem}, the \emph{Sierpi\'nski rank} of a semigroup $S$ is defined to be the least number $n$ such that every countable subset of $S$ is contained in an $n$-generated subsemigroup of $S$. If no such $n$ exists, $S$ is said to have infinite Sierpi\'nksi rank. Note that for countable semigroups, the Sierpi\'nksi rank is just the usual rank of a semigroup, i.e. the least size of a generating set.
It was shown in \cite[Lemma 2.2]{surjections_paper} that the only semigroups of Sierpi\'nski rank $1$ are $1$-generated semigroups. So Theorem \ref{transtheorem} says that $\trans$ has Sierpi\'nski rank $2$.

Sierpi\'nski ranks of various uncountable semigroups have been calculated; see the introduction of \cite{surjections_paper} for a recent survey. The following analogues of Theorem \ref{transtheorem} for groups and inverse semigroups were proved by Galvin and Higgins, Howie, Mitchell and Ru\v{s}kuc, respectively.

\begin{thm}[\cite{galvin}, Theorem 3.3]\label{symtheorem} Let $\Omega$ be an infinite set. Then every countable subset of $\sym$ is contained in a $2$-generated subgroup of $\sym$.
\end{thm}

\begin{thm}[\cite{versus}, Proposition 4.2]\label{invtheorem} Let $\Omega$ be an infinite set. Then every countable subset of $\inv$ is contained in a $2$-generated inverse subsemigroup of $\inv$. \end{thm}

It follows from Theorem \ref{symtheorem} (\ref{invtheorem}) that every countable group (inverse semigroup) embeds in a $2$-generated group (inverse semigroup).

The definition of Sierpi\'nski rank for semigroups extends naturally to general algebras: an algebra $A$ has Sierpi\'nski rank $n$ if every countable subset of $A$ is contained in an $n$-generated subalgebra of $A$. It is easy to see that groups and inverse semigroups of Sierpi\'nski rank $1$ are commutative. So one way of stating Theorems \ref{symtheorem} and \ref{invtheorem} is to say that the group $\sym$ and the inverse semigroup $\inv$ have Sierpi\'nski rank $2$.

Note that the Sierpi\'nski rank of a given object now depends on the type of algebra we choose to view it as. For instance, the Sierpi\'nski rank of an inverse semigroup $S$ may not be the same as the Sierpi\'nski rank
of $S$ seen as an ordinary semigroup that ``just happens to be'' an inverse semigroup. The difference is, of course, that the inverse semigroup generated by some elements of $S$ is the semigroup generated by those elements and their inverses. So the best we can say in general is that the Sierpi\'nski rank of the inverse semigroup $S$ is at most the Sierpi\'nski rank of the semigroup $S$ which, in turn, is at most twice the Sierpi\'nski rank of the inverse semigroup $S$.

There are no such difficulties between groups and inverse semigroups. The Sierpi\'nski rank of a non-trivial group $G$ is the same as the Sierpi\'nski rank of the inverse semigroup $G$. The trivial group has Sierpi\'nski rank $0$ as a group and Sierpi\'nski rank $1$ as an (inverse) semigroup.

Since $\sym$ and $\inv$ are also important and interesting examples in the context of ordinary semigroups, it is natural to ask what their Sierpi\'nski ranks are when seen as semigroups. In the case of $\sym$ the answer is already known. In \cite[Theorem 3.5]{galvin} Galvin showed that the two generators from Theorem \ref{symtheorem} may be taken to have orders $4$ and $53$. In particular, since they have finite orders, the semigroup generated by them is the same as the group generated by them. Hence, seen as a semigroup,  $\sym$ has Sierpi\'nski rank $2$, also.

The purpose of this short note is to prove that the Sierpi\'nski rank of the semigroup $\inv$ is also $2$.
In other words, to prove the following stronger version of Theorem \ref{invtheorem}.

\begin{thm}\label{main} Let $\Omega$ be an infinite set. Then every countable subset of $\inv$ is contained in a $2$-generated subsemigroup of $\inv$.
\end{thm}

\section{Proof of Theorem \ref{main}}

To prove Theorem \ref{main} we require a number of preliminary results. Throughout, we will write $(x)f$ or simply $xf$ for the image of the point $x$ under the function $f$ and compose functions from left to right.

\begin{lem} \label{relative} Let $f,g \in \inv$ such that $\Omega g = \Omega f^{-1}=\Omega$ and $|\Omega\setminus \Omega f|= |\Omega\setminus \Omega g^{-1}|=|\Omega|$. 
Then for every $h \in \inv$ there exists $a \in \sym$ such that $h=fag$.
\end{lem}
\begin{proof} The map $f^{-1}hg^{-1}:\Omega h^{-1}f \longrightarrow \Omega h g^{-1}$ is a bijection. 
Since 
$$|\Omega| \geq |\Omega\setminus \Omega h^{-1} f|\geq |\Omega\setminus \Omega f|=|\Omega|= |\Omega\setminus \Omega g^{-1}|\leq |\Omega\setminus \Omega h g^{-1}|\leq |\Omega|,$$
 we may extend $f^{-1}hg^{-1}$ to $a\in \sym$. Then $fag=f(f^{-1}hg^{-1})g=h$, as required.
\end{proof}


As mentioned earlier, the next result is an immediate consequence of Theorem \ref{invtheorem}. Alternatively, as shown here, it is also a corollary of Theorem \ref{symtheorem}.

\begin{cor}\label{four}
Let $\Omega$ be an infinite set. Then every countable subset of $\inv$ is contained in a $4$-generated subsemigroup of $\inv$.
\end{cor}
\begin{proof}
Let $A$ be an arbitrary countable subset of $\inv$. Let $f,g \in \inv$ satisfy the conditions of Lemma \ref{relative}. Then, by Lemma \ref{relative}, there exists a countable subset $B$ of $\sym$ such that $A\subseteq \genset{f,g,B}$. Since $\sym$ as a semigroup has Sierpi\'nski rank $2$, there exist $h,k\in \sym$ such that $B\subseteq \genset{h,k}$. Thus $A\subseteq \genset{f,g,B}\subseteq \genset{f,g, h,k}$, as required.
\end{proof}

Recall that an element $i$ of $\sym$ is called an \emph{involution} if $i^2$ equals the identity $1_{\Omega}$ on $\Omega$. The following is a well-known result, see, for example, \cite[Lemma 2.2]{galvin}.

\begin{lem}\label{involutions}
Every element of $\sym$ is a product of two involutions.
\end{lem}

We will also require the following result, the proof of which is similar to that of Lemma \ref{involutions}.

\begin{lem}\label{cancel} For every $a\in \sym$ there exists an involution $j \in \sym$ such that $a^{-1} \in \genset{a, aj}$.
\end{lem}
\begin{proof}
Let $\sigma$ be any cycle of $a$ and fix an arbitrary $x$ in the orbit of $\sigma$. 
Define the transformation $j_{\sigma}$ of the orbit $\set{x \sigma^{n}}{n\in \mathbb{Z}}$ of $\sigma$ by 
$(x \sigma^{n})j_{\sigma}=x \sigma^{-n+1}$
for all $n\in \mathbb{Z} $. 
Note that 
$(x \sigma^{-n+1})j_{\sigma}=x \sigma^{-(-n+1)+1}=x \sigma^{n}$
and so $j_{\sigma}$ is an involution on the orbit of $\sigma$.

Furthermore, $(x \sigma^{n})\sigma j_{\sigma}=(x \sigma^{n+1})j_{\sigma}=x \sigma^{-n}$
and so 
$$(x\sigma^n)\sigma j_{\sigma} \sigma \sigma j_{\sigma}=(x\sigma^{-n})\sigma \sigma j_{\sigma}=(x\sigma^{-n+1})\sigma j_{\sigma}=x \sigma^{n-1}.$$
Thus $(\sigma j_{\sigma})\sigma(\sigma j_{\sigma})=\sigma^{-1}$. 

In the same way as above, define $j_{\tau}$ for every cycle $\tau$ of $a$ and let $j$ be the union of all $j_{\tau}$. Then $j\in \sym$ is an involution and 
$(aj)a(aj)=a^{-1}$. In particular, $a^{-1} \in \genset{a, aj}$, as required.
\end{proof}

We are now in a position to prove the main theorem.

\vspace{\baselineskip}
\proofref{main}
By Corollary \ref{four}, it suffices to show that for all $h_1, h_2, h_3,h_4\in \inv$ there exist 
$f,g\in \inv$ such that $h_1,h_2,h_3,h_4 \in \genset{f,g}$.
Partition $\Omega$ into countably infinitely many sets $\Omega_0, \Omega_1, \Omega_2\dots$ where $|\Omega_i|=|\Omega|$ for every $i\in \nat$. 
Let $f$ be any element of $\inv$ that maps $\Omega_i$ bijectively to $\Omega_{i+1}$ for every $i\in \nat$.
Note that $|\Omega\setminus \Omega f|=|\Omega_0|=|\Omega|$ and $\Omega f^{-1}=\Omega$. 

For $13\leq n\leq 22$, let $i_n\in \operatorname{Sym}(\Omega_n)$ be an involution and let $g$ be any element of $\inv$ with domain $\bigcup_{n=13}^{\infty}\Omega_n$ such that:
\begin{itemize}

\item $g|_{\Omega_n}=i_n$ for $13\leq n\leq 22$;

\vspace{1mm}
\item $(\Omega_{23})g=\Omega_{23}\cup \Omega_{24}$;

\vspace{1mm}
\item $(\Omega_{24})g=\bigcup_{n=25}^{\infty} \Omega_n$;

\vspace{1mm}
\item $(\Omega_{25})g=\bigcup_{n=1}^{12}\Omega_n$;

\vspace{1mm}
\item $(\bigcup_{n=26}^{\infty})g=\Omega_0$.
\end{itemize}

The aim is now to specify the involutions $i_n$ in such a way that $h_1,h_2,h_3,h_4 \in \genset{f,g}$. The definition of $i_n$ will depend on $h_1,h_2,h_3,h_4, f$ and $g$. Since $g$, in turn, depends on the $i_n$, we must be very careful to avoid circular definitions.

Note that $g^2$ is independent of the choices for the $i_n$ (as long as every $i_n$ is indeed an involution). Note that the domain of $g^2$ is $(\Omega) g^{-2}=(\bigcup_{n=13}^{\infty})g^{-1}=\bigcup_{n=13}^{24}\Omega_n$ and the range is $\Omega g^2=\Omega g=\Omega$.
Let $\pi=f^{26}g$ and $\tau = g^{-2}f^{-12}g^{-1}f^{-25}$. 
It is easy to verify that $\pi$ and $\tau$ are bijections from $\Omega$ to $\Omega_0$. Furthermore, $\pi$ is independent of the choices for the $i_n$, since $\Omega f^{26}=\bigcup_{n=26}^{\infty}$ has empty intersection with the union $\bigcup_{n=13}^{22}\Omega_n$ of the domains of the $i_n$. 
Similarly, $\tau$ is independent of the choices for $i_n$, since $g^2$, and hence $g^{-2}$, are independent and 
$\Omega g^{-2}f^{-12}=(\bigcup_{n=13}^{24}\Omega_n) f^{-12} =\bigcup_{n=1}^{12}\Omega_n$ has empty intersection with the union $\bigcup_{n=13}^{22}\Omega_n$ of the domains of the $i_n^{-1}$.
In particular, we may, without fear of our argument becoming circular, use $g^2$, $\pi$ and $\tau$ when defining $i_n$.

Since $f$ and $g^2$ satisfy the conditions of Lemma \ref{relative}, there exist $a_1,a_2,a_3,a_4 \in \sym$ such that $h_1,h_2,h_3,h_4 \in \genset{f,g^2, a_1, a_2, a_3, a_4}$. By Lemma \ref{involutions}, there exist involutions $j_1, \dots, j_8\in \sym$ such that 
$a_1,a_2,a_3,a_4 \in \genset{j_1, \dots, j_8}$. Then $h_1,h_2,h_3,h_4 \in \genset{f,g^2, j_1, \dots, j_8}$. Since $\pi$ and $\tau$ are both bijections from $\Omega$ to $\Omega_0$, the composite $\pi \tau^{-1}$ is an element of $\sym$. 
Hence, by Lemma \ref{cancel}, there exists an involution $j_9\in \sym$ such that $(\pi\tau^{-1})^{-1}\in \genset{(\pi\tau^{-1}), (\pi\tau^{-1}) j_9}$. Let $j_{10}$ be the identity $\Omega$. 

Note that $\tau f^{n}$ is a bijection from $\Omega$ to $\Omega_n$ and define 
$$i_n=(\tau f^{n})^{-1} j_{n-12} (\tau f^n)=(f^{-n}\tau^{-1}j_{n-12} \tau f^{n})|_{\Omega_n}$$ 
for $13\leq n \leq 22$. 
Then $i_n$ is an involution since it is the conjugate of the involution $j_{n-12}$.
Furthermore, if $x\in \Omega$ is arbitrary, and $1\leq k \leq 10$, then
$(x) f^{26}g f^{12+k}= (x) \pi f^{12+k}\in \Omega_{12+k}.$ 
Hence 
\begin{eqnarray*}
(x) f^{26}g f^{12+k} g f^{13-k} g f^{12}g^{2}
&=&(x) \pi f^{12+k}(f^{-12-n}\tau^{-1}j_{k} \tau f^{12+k})f^{13-k} g f^{12}g^{2}\\
&=&(x) \pi \tau^{-1} j_{k}\tau f^{25} g f^{12}g^{2}\\
&=&(x) \pi \tau^{-1} j_{k}\tau \tau^{-1}\\
&=&(x) \pi \tau^{-1} j_k.
\end{eqnarray*}
Thus $f^{26}g f^{12+k} g f^{13-k} g f^{12}g^{2}=(\pi\tau^{-1}) j_k$. In particular, $(\pi\tau^{-1}) j_k \in \genset{f,g}$ for $1\leq k\leq 10$. But $j_9$ was chosen such that 
$(\pi\tau^{-1})^{-1}\in \genset{(\pi\tau^{-1}), (\pi\tau^{-1}) j_9}$ and $(\pi\tau^{-1}) =(\pi\tau^{-1}) j_{10}$. Hence $(\pi\tau^{-1})\in \genset{f,g}$. 
It follows that $j_1, \dots, j_8 \in \genset{f,g}$.
Thus 
$$h_1,h_2,h_3,h_4\in \genset{f,g^2, j_1, \dots, j_8}\subseteq \genset{f,g},$$
as required.
\qed


\begin{thebibliography}{99}

\bibitem{galvin}
F. Galvin, \emph{Generating countable sets of permutations}, J. London Math. Soc. 51 (1995), 230Ð242.

\bibitem{versus}P. M. Higgins, J. M. Howie, J. D. Mitchell, and N. Ru\v{s}kuc, \emph{Countable versus uncountable ranks in infinite semigroups of transformations and relations}, Proc. Edinburgh Math. Soc. 46 (2003), 531Ð544.

\bibitem{surjections_paper} J.D. Mitchell and Y. P\'eresse, \emph{Generating countable sets of surjective functions}, Fund. Math. 213 (2011) 67-93;
\bibitem{sierpinski}
W. Sierpi\'nski, \emph{Sur les suites infinies de fonctions d\'efinies dans les ensembles quelconques}, Fund. Math. 24 (1935), 209-212.





\end{thebibliography}
\end{document}